\documentclass[12pt]{amsart}

\usepackage[utf8]{inputenc}
\usepackage[T1]{fontenc}
\usepackage[scaled=0.92]{helvet}    % set Helvetica as the sans-serif font
\makeatletter
\def\@settitle{\begin{center}%
  \baselineskip14\p@\relax
    \normalfont\Large
\uppercasenonmath\@title
  \@title
  \end{center}%
}
\makeatother

\usepackage{amsmath,amssymb,color,setspace}
\usepackage[margin=1in]{geometry}
\usepackage[mathscr]{euscript}

\newcommand{\ga}{\alpha}
\newcommand{\gb}{\beta}
\newcommand{\gr}{\gamma}
\newcommand{\gk}{\kappa}
\newcommand{\gf}{\phi}
\renewcommand{\le}{\leqslant}
\renewcommand{\ge}{\geqslant}
\newcommand{\vm}{\vspace{6pt}}
\newcommand{\li}{\text{li}}
\newcommand{\gO}{\Omega}

\newcommand{\Mod}[1]{\ (\mathrm{mod}\ #1)}

\newtheorem{theorem}{Theorem}
\newtheorem*{theorem*}{Theorem}
\newtheorem{lemma}{Lemma}[section]

\title{Almost-prime values of reducible polynomials at prime arguments}
\author{C.\ S.\ Franze}
\address{Department of Mathematics, The Ohio State University}
\email{\texttt{franze.3@osu.edu}}

\author{P.\ H.\ Kao}
\address{Department of Mathematics and Technology, Flagler College}
\email{\texttt{ckao@flagler.edu}}
\date{}  

\begin{document}

\maketitle

\begin{abstract}
We adopt A.\ J.\ Irving's sieve method to study the almost-prime values produced by products of irreducible polynomials evaluated at prime arguments. This generalizes the previous results of Irving and Kao, who separately examined the almost-prime values of a single irreducible polynomial evaluated at prime arguments.
\end{abstract}

\section{Introduction}

In this paper, we adopt a sieve method developed by A.\ J.\ Irving in \cite{Irv} to prove
\begin{theorem}\label{Th:1}
    \normalfont
    Let $H(n) = h_1(n) \cdots h_g(n)$, where $h_i$ are distinct irreducible polynomials each with integer coefficients and $\deg h_i = k$ for all $i = 1, \ldots, g$. Suppose that
    \[
        \#\{ a \Mod{p} : (a,p) = 1 \text{ and } H(a) \equiv 0 \Mod{p} \} < p-1.
    \]
    Then, for sufficiently large $x$, there exists a natural number $r$ such that
    \begin{equation}\label{P_r_bound}
        \sum_{\substack{ x<p\le 2x \\ \Omega(H(p))\le r}} 1 \gg \frac{x}{\log^{g+1} x}.
    \end{equation}
    If $g\ge 2$ and $k$ is sufficiently large, we may select an $r$ of the form
    \begin{equation}\label{r_admissible_theorem_version}
        r = gk + c_1 g^{3/2} k^{1/2} + c_2 g^2 + O\left( g\log gk \right),
    \end{equation}
    where $c_1$ and $c_2$ are $O(1)$. Explicit admissible values of $r$ for small $g$ and $k$ are given below.
    \begin{center}
\begin{table}[h]
\begin{tabular}{|c|r|r|r|r|r|r|r|r|r|r|r|r|r|r|}\hline
    $g \setminus k$ & 1 & 2 & 3 & 4 & 5 & 6 & 7 & 8 & 9 & 10 & 11 & 12 & 13 & 14\\ \hline\hline
    2 & -- & -- & 15 & 18 & 21 & 23 & 26 & 29 & 31 & 33 & 36 & 38 & 40 & 43\\ \hline
    3 & -- & -- & -- & 30 & 35 & 39 & 43 & 47 & 51 & 55 & 59 & 62 & 66 & 70\\ \hline
    4 & -- & -- & -- & 43 & 50 & 56 & 63 & 68 & 74 & 79 & 85 & 90 & 95 & 100\\ \hline
\end{tabular}
\vm
\caption{Admissible values for $r$ using Irving's sieve method}
\label{T:2}
\end{table}
\end{center}
\end{theorem}

The case $g=1$ was first investigated by H.-E.\ Richert in 1969~\cite{Ric}, who showed that for each $k\ge 1$, $r=2k+1$ is an admissible choice. Virtually no progress was made until Irving's work in 2015~\cite{Irv}, which showed that one could take an $r$ of the form $r=k+O(\log k)$. Explicit bounds for the $O$-term, as well as explicit values for $r$ when $k$ is small, are available in~\cite{Irv} and~\cite{Kao}.

The more general case where $g\ge 2$ is studied in the book by Halberstam and Richert in 1974~\cite{HR}, who showed that one could select an $r$ of the form
\begin{equation}\label{classical_r}
    r = 2gk + O(g \log gk).
\end{equation}
Their method was refined in the book by Diamond and Halberstam~\cite{DH2}, which offers the admissible $r$ described below in Table~\ref{T:1} (see \cite[pp.149-150]{DH2}). However, their admissible $r$ exhibit the same asymptotic behavior described in \eqref{classical_r}. Therefore, the results of Theorem~\ref{Th:1} represent an improvement when $k\gg g$.
\begin{center}
\begin{table}[h]
\begin{tabular}{|c|r|r|r|r|r|r|r|r|r|r|r|r|r|r|}\hline
    $g \setminus k$ & 1 & 2 & 3 & 4 & 5 & 6 & 7 & 8 & 9 & 10 & 11 & 12 & 13 & 14\\ \hline\hline
    2 & 7 & 11 & 16 & 20 & 24 & 28 & 32 & 36 & 40 & 44 & 48 & 52 & 56 & 60\\ \hline
    3 & 12 & 19 & 25 & 32 & 38 & 44 & 50 & 56 & 62 & 69 & 75 & 81 & 87 & 93\\ \hline
    4 & 17 & 27 & 35 & 44 & 52 & 61 & 69 & 77 & 86 & 94 & 102 & 110 & 118 & 126\\ \hline
\end{tabular}
\vm
\caption{Classical admissible values for $r$}
\label{T:1}
\end{table}
\end{center}
Irving's innovation was to combine a linear (one-dimensional) sieve with a two-dimensional sieve that permits a level of distribution beyond that which is available using the Bombieri-Vinogradov theorem. We adopt this novel idea to the relevant $g$- and $g+1$-dimensional sieves used for the more general polynomial sequence, $H$, considered here. The sifting functions $F_g$ and $f_g$ are, however, more difficult to work with for $g\ge 2$.

\section{Main Sieve Setup}

Here, we adopt some standard sieve notation. Setting $P(z)=\prod_{p<z}p$, we require bounds on 
\begin{equation}\label{survivors}
S(A,z)=\#\{n\in A: (n,P(z))=1\}.
\end{equation}
The sequence that we are going to sieve is
\[
    A = \{ H(p): x<p \le 2x \}.
\]
Using the prime number theorem, we note that the cardinality $|A|\sim X$, where  
\begin{equation}\label{X_def}
    X=\li x.
\end{equation}
Letting $A_d=\{n\in A: n \equiv 0\ (d)\}$, it is straightforward (e.g. see \cite[pp.131-132]{DH2}) to show that
\begin{equation}\label{A_d}
    |A_d| = \frac{\rho_1(d)}{\phi(d)}X +r_A(d),
\end{equation}
where
\[
    \rho_1(d) := \#\{ a \Mod{d} : 1\le a \le d \text{ and }  (a,d)=1 \text{ and } H(a) \equiv 0 \Mod{d} \},
\]
and the remainder term, $r_A(d)$, is bounded by
\[
    |r_A(d)|\le \rho(d)E(x,d)+\rho(d),
\]
where
\[
    \rho(d) := \#\{ a \Mod{d} : 1\le a \le d \text{ and } H(a) \equiv 0 \Mod{d} \},
\]
and
\begin{equation*}
    E(x,d)=\max_{\substack{1\le m\le d \\ (m,d)=1}}\left|\pi(x,d,m)-\frac{\li x}{\phi(d)}\right|.
\end{equation*}
The sieve dimension is $g$ since the density function $\rho_1(d)/\phi(d)$ appearing in \eqref{A_d} satisfies
\begin{equation}\label{density_hypothesis_g}
    \sum_{p\le x}\frac{\rho_1(p)}{\phi(p)}\log p=g\log x+O(1).
\end{equation}
This follows from Proposition 10.1 of \cite{DH2}, which gives
\begin{equation}\label{density_hypothesis_g_simple}
    \sum_{p\le x}\frac{\rho_1(p)}{p}\log p=g\log x+O(1),
\end{equation}
since
\begin{equation*}
    \sum_{p\le x}\left(\frac{\rho_1(p)}{\phi(p)}-\frac{\rho_1(p)}{p}\right)\log p \ll \sum_{p\le x}\frac{\rho_1(p)}{p^2}\log p \ll_H \sum_{p\le x}\frac{\log p}{p^2}\ll 1, 
\end{equation*}
where we used $\rho_1(p)\le\rho(p)\le \deg(H)$.

As a consequence of \eqref{density_hypothesis_g}, the product
\begin{equation}\label{V_def}
    V(z):=\prod_{p<z}\left(1-\frac{\rho_1(p)}{\phi(p)}\right)\gg (\log z)^{-g}.
\end{equation}

Finally, we note that the Bombieri-Vinogradov theorem implies that for any $\tau_1\le\frac{1}{2}$,
\begin{equation}\label{BV}
    \sum_{\substack{d \text{squarefree} \\ d < X^{\tau_1}(\log X)^{-B}}} 4^{\omega(d)}|r_A(d)| \ll  \frac{X}{(\log X)^{g+1}},
\end{equation}
for a suitably large value of $B$ (e.g. see \cite[Lemma 3.5 on p.115, and p. 288]{HR}). The parameter $\tau_1$ is called the level of distribution.

\section{An Auxiliary Sieve}

The main difference between Irving's approach, adopted here, and the classical one is the introduction of an auxiliary upper bound sieve for the sequence $A_p$, where $p$ is a prime $z\le p<y$. Recall from \eqref{survivors} that
\begin{equation*}
    S(A_p,z)=\sum_{\substack{x<q\le 2x \\ q \text{\ prime} \\ H(q)\equiv 0 (p) \\ (H(q),P(z))=1}}1.
\end{equation*}
If $z<x$, then for any prime $q>x$ we plainly have $(q,P(z))=1$. Therefore,
\begin{equation}\label{Irving_Switch}
    S(A_p,z)=\sum_{\substack{x<q\le 2x \\ q \text{\ prime} \\ H(q)\equiv 0 (p) \\ (qH(q),P(z))=1}}1 \le \sum_{\substack{x<n\le 2x \\ H(n)\equiv 0 (p) \\ (nH(n),P(z))=1}}1 = S(A',z),
\end{equation}
where
\[
    A'=\{ nH(n): x<n\le 2x, p\mid H(n)\}.
\]
Although the upper bound available for $S(A',z)$ is worse than that for $S(A_p,z)$, a larger level of distribution is available to us for $A'$, which involves \emph{integer} arguments rather than primes. In this case, the cardinality $|A'|\sim X'$, where
\begin{equation*}
    X'=\frac{\rho_1(p)}{p}x,
\end{equation*}
and, using the Chinese remainder theorem, we observe that
\begin{equation}\label{A'_d}
    |A'_d| = \frac{\rho_2(d)}{d} X' +r_{A'}(d),
\end{equation}
where
\begin{equation*}
    \rho_2(d):= \#\{ a \Mod{d} : aH(a) \equiv 0 \Mod{d} \},
\end{equation*}
and the remainder term, $r_{A'}(d)$, is bounded by
\begin{equation}\label{r12}
    |r_{A'}(d)|\le \rho_1(p)\rho_2(d),
\end{equation}
for $d\mid P(z)$ and $p\ge z$ large enough to ensure that $p\nmid H(0)$ (see proof of Lemma 4.2 in \cite{Irv}). The sieve dimension is $g+1$ in this case since the density function $\rho_2(d)/d$ appearing in \eqref{A'_d} satisfies
\begin{equation}\label{density_hypothesis_g+1}
    \sum_{p\le x}\frac{\rho_2(p)}{p}\log p=(g+1)\log x+O(1),
\end{equation}
owing to the fact that $\rho_2(p)=\rho_1(p)+1$.
As a consequence of \eqref{density_hypothesis_g+1}, we have
\begin{equation}\label{V'_def}
    V'(z):=\prod_{p<z}\left(1-\frac{\rho_2(p)}{p}\right)\gg (\log z)^{-(g+1)}.
\end{equation}
More precisely, using Mertens' product formula,
\begin{equation}\label{Merten_estimate}
    V'(z)=\prod_{p<z}\left(1-\frac{1}{p}\right)\left(1-\frac{\rho_1(p)}{\phi(p)}\right)\sim\frac{e^{-\gamma}}{\log z}V(z).
\end{equation}
Using \eqref{X_def}, we note that
\[
    x \sim X\log X,
\]
and therefore,
\begin{equation}\label{X'_to_X}
    X' \sim \frac{\rho_1(p)}{p}X\log X.
\end{equation}
In contrast to \eqref{BV}, upon setting $z=X^{1/v}$, a small power of $X$, we see that for any $\tau_2\le 1$,
\begin{equation}\label{BV_improvement}
    \sum_{\substack{d \mid P(z) \\ pd<X^{\tau_2}(\log X)^{-B'}}}4^{\omega(d)}|r_{A'}(d)| = o\left( X'V'(z)\right),
\end{equation}
for a suitably large $B'$. This is easily obtained using \eqref{r12} and \eqref{V'_def} so that
\begin{equation*}
    \sum_{\substack{ d\mid P(z) \\ pd<X^{\tau_2}(\log X)^{-B'} }} 4^{\omega(d)} |r_{A'}(d)| \le \frac{\rho_1(p)}{p}X^{\tau_2}(\log X)^{-B'}\sum_{d\mid P(z)} \frac{4^{\omega(d)} \rho_2(d)}{d}.
\end{equation*}
Proceeding in the manner of the proof of Lemma 4.3 in \cite{DH2}, we conclude that this is
    \begin{equation*}
        X'\left(\frac{X^{\tau_2}}{x}\right)(\log X)^{-B'}\prod_{p<z}\left(1+\frac{4\rho_2(p)}{p}\right)
        \ll X'\left(\frac{\li x}{x}\right)(\log X)^{-B'}V'(z)^{-4}=o\left(X'V'(z)\right),
    \end{equation*}
for a suitably large $B'$.

\section{Diamond-Halberstam-Richert Sieve}

We will employ the Diamond-Halberstam-Richert (DHR) sieve to estimate the number of survivors, $S(A,z)$, $S(A_p,z)$, and $S(A',z)$. Recall from Theorem 9.1 of~\cite{DH2} that for any $2\le z\le y$,
\begin{equation}\label{DHR_upper}
    S(A,z) \le XV(z) \left( F_g \left(\frac{\log y}{\log z}\right) +O\left(\frac{(\log\log y)^2}{(\log y)^{1/(2g+2)}}\right) \right) + 2\sum_{\substack{ m \mid P(z) \\ m < y}} 4^{\omega(m)}|r_A(m)|,
\end{equation}    
and,
\begin{equation}\label{DHR_lower}
    S(A,z) \ge XV(z) \left( f_g \left(\frac{\log y}{\log z}\right) -O\left(\frac{(\log\log y)^2}{(\log y)^{1/(2g+2)}}\right) \right) - 2\sum_{\substack{ m \mid P(z) \\ m < y}} 4^{\omega(m)}|r_A(m)|.
\end{equation}
The functions $F_g$ and $f_g$ are defined by the unique solutions to the differential-delay equations
\begin{align}
  \left(u^g F_g(u)\right)'&=g u^{g-1}f_g(u-1),\quad u>\alpha_g\label{F_diffeq}\\
  \left(u^g f_g(u)\right)'&=g u^{g-1}F_g(u-1),\quad u>\beta_g,\label{f_diffeq}
\end{align}
with initial conditions
\begin{align*}
  F_g(u)&=\frac{1}{\sigma_g(u)},\quad 0<u\le\alpha_g,\\
  f_g(u)&=0,\quad 0<u\le\beta_g,
\end{align*}
where $\sigma_g$ is the Ankeny-Onishi function, and
\begin{equation*}
  \alpha_1=\beta_1=2\quad\text{and}\quad\alpha_g>\beta_g>2\quad\text{for $g>1$}.
\end{equation*}
We suppose here that $g$ is a positive integer, and remark that Booker and Browning \cite{BB} have recently compiled a list of values for $\alpha_g$ and $\beta_g$ for $g\le 50$. The sifting limit $\beta_g$ satisfies $\beta_g\lesssim cg$, where $c\approx 2.445$ (see \cite[Theorem 2]{DH1}, and \cite{AO}).  The functions $F_g$ and $f_g$ satisfy
\begin{equation}\label{F_f_boundary_condition}
  F_g(u)=1+O\left(e^{-u}\right),\quad f_g(u)=1+O\left(e^{-u}\right),
\end{equation}
and $F_g$ decreases monotonically, while $f_g$ increases monotonically on $(0,\infty)$. In fact, Diamond and Halberstam establish in \cite[Lemma 6.2]{DH2} that for $1\le u_1 < u_2$,
\begin{equation}\label{F_lipschitz}
    0 \le F_{g}(u_1)-F_{g}(u_2)\le \frac{u_2-u_1}{u_1} \cdot \frac{g}{\sigma_{g}(1)},
\end{equation}
and
\begin{equation}\label{f_lipschitz}
    0 \le f_{g}(u_2)-f_{g}(u_1)\le \frac{u_2-u_1}{u_1} \cdot \frac{g}{\sigma_{g}(1)}.
\end{equation}

\section{Richert Weights}

The aforementioned DHR sieve is enhanced by incorporating certain weights introduced by Richert~\cite{Ric}. The arithmetic significance of these weights are summarized in the lemma below. 
\begin{lemma}\label{Richert_Weight_Lemma}
Suppose $y=X^{1/u}$, $z=X^{1/v}$, and $0<\frac{1}{v}<\frac{1}{u}<\tau_2\le 1$. Let $r$ be a natural number such that $r+1>gku$, and define $\eta:=r+1-gku$. Then for $x$ sufficiently large,
\begin{equation}\label{simple_Booker_Browning_sum}
    \sum_{\substack{n\in A \\ \Omega(n)\le r \\ (n,P(z))=1}} 1 \ge \frac{1}{r+1}W(A)-o(XV(z)),
\end{equation}
where
\begin{equation}\label{W_def}
     W(A) := \sum_{\substack{n \in A\\ (n,P(z))=1}} \biggl( \eta - \sum_{\substack{z \le p < y\\ p|n}} \biggl( 1 - \frac{\log p}{\log y} \biggr) \biggr).
\end{equation}
\end{lemma}
Thus, if we can show that the weighted sum $W(A)$ remains large even as $x$ grows large, say for example $W(A)\gg XV(z)$, then we succeed in demonstrating the abundance of elements $n\in A$ which contain at most $r$ prime factors. The proof of this lemma is contained in \cite[pp.140-141]{DH2}. We briefly reproduce it here for completeness.
\begin{proof}
We begin by observing that the number of elements $n \in A$ that are divisible by $p^2$ for a $z \le p < y$ is negligible. More specifically,
\begin{equation*}
    \sum_{z\le p<y} |A_{p^2}| \ll \sum_{z\le p<y}\rho(p^2)\left(\frac{x}{p^2}+O(1)\right) \ll_{H} \frac{x}{z} + y = o(XV(z)),
\end{equation*}
since $\rho(p^2) \le \- \deg(H) D^2$, where $D$ is the discriminant of $H$~\cite[p.\ 260]{HR}. Therefore, we have
\begin{equation}\label{Squarefree_Pass}
    W(A)=W(A^{\ast})+o(XV(z)),
\end{equation}
where
\[
    A^{\ast}:=A\setminus\bigcup_{z\le p<y}A_{p^2}.
\]
If an $n \in A^{\ast}$ contains a repeated prime factor $p$, then $p \ge y$, and so
\begin{equation}\label{omega_inequality}
    \sum_{\substack{z \le p < y\\ p|n}} \left( 1 - \frac{\log p}{\log y} \right) \ge \sideset{}{^\ast} \sum_{\substack{p \ge y\\ p|n}} \left( 1 - \frac{\log p}{\log y} \right) = \gO(n) - \frac{\log |n|}{\log y}\ge \Omega(n)-\frac{\log X^{gk}}{\log X^{1/u}},
\end{equation}
where $\sum^\ast$ denotes summation over the appropriate multiplicity. It follows from \eqref{W_def} and \eqref{omega_inequality} that
\begin{equation*}
    W(A^{\ast}) \le \sum_{\substack{n \in A^{\ast}\\ (n,P(z))=1}} (r+1-\Omega(n))\le\sum_{\substack{n \in A^{\ast}\\ \Omega(n)\le r \\ (n,P(z))=1}} (r+1).
\end{equation*}
Combining this inequality with \eqref{Squarefree_Pass} finishes the proof of the lemma since
\begin{equation*}
    \sum_{\substack{n\in A \\ \Omega(n)\le r \\ (n,P(z))=1}} 1 \ge \sum_{\substack{n\in A^{\ast} \\ \Omega(n)\le r \\ (n,P(z))=1}} 1 \ge \frac{1}{r+1}W(A)-o(XV(z)). \qedhere
\end{equation*}
\end{proof}
The observant reader may note that $gk$ should be replaced with $gk+\varepsilon$ in \eqref{omega_inequality} since
\begin{equation*}
    \max_{n \in A^{\ast}} |n| \le X^{gk+\varepsilon},
\end{equation*}
for $x$ sufficiently large. The presence of this $\varepsilon$, however, makes little difference in the final analysis.

\section{Approximating the Weighted Sum}

In this section, we turn our attention to approximating the weighted sum, $W(A)$, by integrals. Recall that $z=X^{1/v}$ and $y=X^{1/u}$. Letting $s \in (z,y)$, say $s=X^{1/w}$, we have
\[
    W(A) = \eta S(A,z) - \left(S_1+S_2\right),
\]
where
\[
    S_1 := \sum_{z \le p < s} \left(1 - \frac{\log p}{\log y} \right) S(A_p,z),
\]  
and,
\[
    S_2 := \sum_{s \le p < y} \left(1 - \frac{\log p}{\log y} \right) S(A_p,z).
\]
For $S(A,z)$ and $S_1$, we invoke the Bombieri-Vinogradov theorem in \eqref{BV} for the underlying $g$-dimensional sieve. However, for $S_2$, we will swap $S(A_p,z)$ for $S(A',z)$, where we can instead make use of \eqref{BV_improvement} for the underlying $(g+1)$-dimensional sieve. For readers who wish to skip ahead, we are ultimately lead to an integral form for $W(A)$ stated below in Lemma~\ref{W_integral_form}. The following three lemmas provide the necessary bounds for $S(A,z)$, $S_1$, and $S_2$.
\begin{lemma}\label{S_A_z_lemma}
    Let $z=X^{1/v}$, and $0<\frac{1}{v}<\tau_1\le\frac{1}{2}$. Then
\begin{equation*}
    S(A,z) \ge XV(z) \left\{ f_g\left(\tau_1 v \right) - o\left( 1 \right)\right\}.
\end{equation*}
\end{lemma}

\begin{proof}
Letting $y=X^{\tau_1}(\log X)^{-B}$, $X= \li x$, we conclude at once from \eqref{V_def}, \eqref{BV}, and \eqref{DHR_lower} that
\[
    S(A,z) \ge XV(z) \left\{ f_g\left(\tau_1 v - Bv \frac{\log \log X}{\log X}\right) - O\left( \frac{(\log\log X)^2}{(\log X)^{1/(2g+2)}} \right)\right\}-o\left( XV(z) \right).
\]
Finally, equation \eqref{f_lipschitz} allows us to perturb the argument of $f_g$ at a small expense, so that
\[
    f_g\left(\tau_1 v - Bv \frac{\log \log X}{\log X}\right)\ge f_g\left( \tau_1 v \right) - O\left( \frac{\log\log X}{\log X} \right). \qedhere
\]
\end{proof}

\begin{lemma}\label{L:S1}
    Let $z=X^{1/v}$, $s=X^{1/w}$, and $y=X^{1/u}$ where $0<\frac{1}{v}<\frac{1}{w}<\tau_1\le\frac{1}{2}<\frac{1}{u}$. Then
    \[
        S_1 \le XV(z) g \left\{ \int_w^v \left( 1 - \frac{u}{t} \right) F_g\left(v\left(\tau_1 - \frac{1}{t}\right) \right) \, \frac{dt}{t} + o(1) \right\}.
    \]
\end{lemma}

\begin{proof}
    We apply the $g$-dimensional upper bound DHR sieve in \eqref{DHR_upper} to $S(A_p,z)$ with level of distribution $X^{\tau_1}/p$. Letting $z=X^{1/v}$, $y=X^{\tau_1}(\log X)^{-B}/p$ in \eqref{DHR_upper}, we have
    \[
        \frac{(\log\log y)^2}{(\log y)^{1/(2g+2)}}\ll \frac{(\log\log X)^2}{(\log(X^{\tau_1}(\log X)^{-B}/p))^{1/(2g+2)}}\ll\frac{(\log\log X)^2}{(\log(X^{\tau_1-\frac{1}{w}}(\log X)^{-B}))^{1/(2g+2)}},
    \]
    and so,
    \begin{equation*}
        S(A_p,z) \le \frac{\rho_1(p)}{\gf(p)} XV(z) \left( F_g\left(\frac{\log (X^{\tau_1}(\log X)^{-B}/p)}{\log X^{1/v}} \right)+o(1) \right) + 2 \sum_{m\in\mathcal{M}_p} 4^{\omega(m)}|r_{A_p}(m)|,
    \end{equation*}
    where
    \[
        \mathcal{M}_p:=\{ m|P(z): m<X^{\tau_1}(\log X)^{-B}/p \}.
    \]
    Applying \eqref{F_lipschitz} to perturb the argument of $F_g$ at a small expense, we have 
    \begin{equation*}
        S(A_p,z) \le \frac{\rho_1(p)}{\gf(p)} XV(z) \left( F_g\left(\tau_1 v-v\frac{\log p}{\log X} \right) + o(1) \right) + 2 \sum_{m\in\mathcal{M}_p} 4^{\omega(m)}|r_A(pm)|.
    \end{equation*}
    Now, summing over $p$ in $S_1$, we have
    \begin{equation}\label{S_1_UN}
        S_1 \le XV(z) \sum_{z\le p<s} \left( 1 - \frac{\log p}{\log y} \right)\frac{\rho_1(p)}{\phi(p)} \left( F_g\left(\tau_1 v-v\frac{\log p}{\log X} \right) + o(1) \right) + o\left(XV(z)\right),
    \end{equation}
    since, by the Bombieri-Vinogradov in \eqref{BV},
    \[
        \sum_{z\le p<s}\sum_{m\in\mathcal{M}_p} 4^{\omega(m)}|r_A(pm)|\ll \sum_{\substack{n<X^{\tau_1}(\log X)^{-B}\\ n\ \text{squarefree}}} 4^{\omega(n)}|r_A(n)|=o\left(XV(z)\right).
    \]
    Using \eqref{density_hypothesis_g}, and recalling that $z=X^{1/v}$, and $s=X^{1/w}$, we find that
    \begin{equation*}
        \sum_{z\le p<s}\frac{\rho_1(p)}{\phi(p)}\ll g \log\left(\frac{\log s}{\log z}\right)\ll g \log\frac{v}{w} \ll 1.
    \end{equation*}
    Therefore, distributing the sum in \eqref{S_1_UN} gives
    \begin{equation*}
        S_1 \le XV(z) \left( \sum_{z\le p<s} \left( 1 - \frac{\log p}{\log y} \right)\frac{\rho_1(p)}{\phi(p)} F_g\left(\tau_1 v-v\frac{\log p}{\log X} \right) + o\left( 1 \right) \right).
    \end{equation*}
    Passing from this sum to the stated integral is a standard exercise in Riemann-Stieltjes integration, or summation by parts. For example, we may write the sum as
    \begin{equation}\label{RS_rudimentary}
        \int_{z^{-}}^{s}\left(1-\frac{\log T}{\log y}\right)F_g\left(\tau_1 v-v\frac{\log T}{\log X}\right)\frac{dS(T)}{\log T},
    \end{equation}
    with
    \[
        S(T)=\sum_{p\le T}\frac{\rho_1(p)}{\phi(p)}\log p.
    \]
    If $z=X^{1/v}$, $s=X^{1/w}$, $y=X^{1/u}$, then \eqref{density_hypothesis_g} implies that the integral in \eqref{RS_rudimentary} is asymptotic to
    \[
        g\int_{X^{1/v}}^{X^{1/w}}\left(1-\frac{\log T}{\log X^{1/u}}\right)F_g\left(v\left(\tau_1-\frac{\log T}{\log X}\right)\right)\frac{d\log T}{\log T},
    \]
    Performing the change of variables $T=X^{1/t}$ finishes the proof.
\end{proof}

\begin{lemma}\label{L:S2}
    Let $z=X^{1/v}$, $s=X^{1/w}$, $y=X^{1/u}$ where $0<\frac{1}{v}<\frac{1}{w}<\tau_1\le\frac{1}{2}<\frac{1}{u}<\tau_2 \le 1$.
    \[
        S_2 \le XV(z) \frac{gv}{e^\gr} \left\{ \int_u^w \left( 1 - \frac{u}{t} \right) F_{g+1}\left( v \left( \tau_2 - \frac{1}{t} \right) \right) \frac{dt}{t} + o(1) \right\}.
    \]
\end{lemma}

\begin{proof}
    Here we use \eqref{Irving_Switch} to swap $S(A_p,z)$ for $S(A',z)$, since
    \begin{equation*}
        S(A_p,z)\le S(A',z),
    \end{equation*}
    and then apply the $(g+1)$-dimensional upper bound DHR sieve in \eqref{DHR_upper} to $S(A',z)$, with $X$ replaced by $X'$, $V(z)$ replaced by $V'(z)$, $z=X^{1/v}$, and $y=X^{\tau_2}(\log X)^{-B'}/p$ for a suitably large $B'$. Using \eqref{BV_improvement} to control the remainder term gives
    \begin{equation*}
        S(A',z)\le X'V'(z)\left( F_{g+1} \left( \frac{\log\left( X^{\tau_2}(\log X)^{-B'}/p\right)}{\log X^{1/v}} \right) + o(1) \right).
    \end{equation*}
    Appealing to \eqref{F_lipschitz} to perturb the argument of $F_{g+1}$ so that
    \begin{equation*}
        F_{g+1}\left( \frac{\log\left( X^{\tau_2}/p\right)}{\log X^{1/v}}-B'v\frac{\log\log X}{\log X} \right)\le F_{g+1} \left( \frac{\log\left( X^{\tau_2}/p\right)}{\log X^{1/v}} \right) + O\left(\frac{\log\log X}{\log X}\right),
    \end{equation*}
    gives
    \begin{equation*}
        S(A',z)\le X'V'(z)\left( F_{g+1} \left( \frac{\log\left( X^{\tau_2}/p\right)}{\log X^{1/v}} \right) + o(1) \right).
    \end{equation*}
    Replacing $V'(z)$ and $X'$ with their corresponding expressions in \eqref{Merten_estimate} and \eqref{X'_to_X},
    \begin{equation*}
        S(A',z)\le XV(z)\frac{\rho_1(p)}{p}e^{-\gamma}\frac{\log X}{\log z}\left( F_{g+1} \left( v\tau_2 - v\frac{\log p}{\log X} \right) + o(1) \right).
    \end{equation*}
    Summing over $s\le p<y$ in $S_2$ then gives
    \begin{equation}\label{S_2_UN}
        S_2\le XV(z) e^{-\gamma}v\left(\sum_{s\le p<y}\left(1-\frac{\log p}{\log y}\right)\frac{\rho_1(p)}{p}F_{g+1}\left(v\tau_2 - v\frac{\log p}{\log X}\right)+o(1)\right),
    \end{equation}
    since \eqref{density_hypothesis_g_simple} implies that
    \begin{equation*}
        \sum_{s\le p<y}\frac{\rho_1(p)}{p}\ll g \log\left(\frac{\log y}{\log s}\right)\ll g \log\frac{w}{u} \ll 1.
    \end{equation*}
    Passing from the sum in \eqref{S_2_UN} to the stated integral is a standard exercise. Note that this sum is 
    \begin{equation}\label{RS_rudimentary_2}
        \int_{s^{-}}^{y}\left(1-\frac{\log T}{\log y}\right)F_g\left(\tau_2 v-v\frac{\log T}{\log X}\right)\frac{dS(T)}{\log T},
    \end{equation}
    with
    \[
        S(T)=\sum_{p\le T}\frac{\rho_1(p)}{p}\log p.
    \]
    Recalling that $s=X^{1/w}$, $y=X^{1/u}$, and using \eqref{density_hypothesis_g_simple}, the integral in \eqref{RS_rudimentary_2} is asymptotic to
    \[
        g\int_{X^{1/w}}^{X^{1/u}}\left(1-\frac{\log T}{\log X^{1/u}}\right)F_{g+1}\left(v\left(\tau_2-\frac{\log T}{\log X}\right)\right)\frac{d\log T}{\log T},
    \]
    Performing the change of variables $T=X^{1/t}$ finishes the proof.
\end{proof}

Combining Lemma~\ref{S_A_z_lemma}, Lemma~\ref{L:S1}, and Lemma~\ref{L:S2} gives
\begin{lemma}\label{W_integral_form}
Let $0 < \frac{1}{v} < \frac{1}{w} < \tau_1 \le \frac{1}{2} < \frac{1}{u} < \tau_2 \le 1$. Then
\[
    W(A)\ge\left(\eta f_g(\tau_1 v)-\left(I(u,w,v)+\frac{v}{e^{\gamma}}J(u,w,v)\right)+o(1)\right)XV(z),
\]
where
\begin{equation}\label{I_def}
  I(u,w,v):=g\int_{w}^{v}\left(1- \frac{u}{s}\right)F_{g}\left(v \left(\tau_1- \frac{1}{s}\right) \right)\frac{ds}{s},
\end{equation}
and,
\begin{equation}\label{J_def}
  J(u,w,v):=g\int_{u}^{w}\left(1- \frac{u}{s}\right)F_{g+1}\left(v \left(\tau_2- \frac{1}{s}\right) \right)\frac{ds}{s}.
\end{equation}
\end{lemma}

\section{Simple Estimates for the Integrals}

Analysis for higher dimensional sieves is obstructed by the evaluation of $I:=I(u,w,v)$ and $J:=J(u,w,v)$, appearing above in \eqref{I_def} and \eqref{J_def}. Useful estimates of these integrals are presented below. Analysis closely follows Section 11.4 of Diamond-Halberstam~\cite{DH2}.
\begin{lemma}\label{Main_I_Lemma}
Let $\xi_1:=v\tau_1+1-\frac{v}{w}$, and  $0<\frac{1}{v}<\frac{1}{w}<\tau_1\le\frac{1}{2}<\frac{1}{u}<\tau_2\le 1$. If $\xi_1\ge\beta_g$, then
  \begin{equation*}
    \frac{1}{f_g(\tau_1 v)}I\le \left(g+\frac{u}{v}\xi_1 \left(1-\frac{f_g(\xi_1)}{f_g(\tau_1 v)}\right)\right)\log\frac{v}{w}+\left(1- \frac{f_g(\xi_1)}{f_g(\tau_1 v)}\right)\xi_1 \frac{w}{v}\left(1- \frac{u}{w}\right)-g \left(\frac{u}{w}-\frac{u}{v}\right).
  \end{equation*}
\end{lemma}
\begin{proof}
Let $t-1=v(\tau_1-\frac{1}{s})$, so $s=v/(v\tau_1+1-t)$. Under this change of variables,
\begin{equation*}
  I=g \frac{u}{v}\int_{\xi_1}^{v\tau_1}F_g(t-1)\frac{t-\xi_1+\frac{v}{u}-\frac{v}{w}}{v\tau_1+1-t}dt.
\end{equation*}
We then separate the integral so that
\begin{equation}\label{I_sum}
  I=I_1+g \frac{u}{v}\left(\frac{v}{u}-\frac{v}{w}\right)I_2,
\end{equation}
where
\begin{equation*}
  I_1:=g \frac{u}{v}\int_{\xi_1}^{v\tau_1}F_g(t-1)\frac{t-\xi_1}{v\tau_1+1-t}dt,
\end{equation*}
and
\begin{equation*}
  I_2:=\int_{\xi_1}^{v\tau_1}F_g(t-1)\frac{dt}{v\tau_1+1-t}.
\end{equation*}
Integrating by parts,
\begin{align*}
  I_1&=-g\frac{u}{v}\int_{\xi_1}^{v\tau_1}F_g(t-1)(t-\xi_1)d\left(\log(v\tau_1+1-t)\right)\\
    &=g\frac{u}{v}\int_{\xi_1}^{v\tau_1}\left(F_g(t-1)+F'_g(t-1)(t-\xi_1)\right)\log(v\tau_1+1-t)dt\\
    &<g\frac{u}{v}\int_{\xi_1}^{v\tau_1}F_g(t-1)\log(v\tau_1+1-t)dt,
\end{align*}
since $F$ is decreasing. Next, if $\xi_1\ge\beta_g$, then $t\ge\beta_g$, and we can use \eqref{f_diffeq} to observe that
\begin{equation*}
  I_1<\frac{u}{v}\int_{\xi_1}^{v\tau_1}\left(t^gf_g(t)\right)'t^{1-g}\log(v\tau_1+1-t)dt.
\end{equation*}
Integrating by parts, and using the fact that $f$ is increasing, gives
\begin{align*}
    I_1&<\frac{u}{v}\int_{\xi_1}^{v\tau_1}f_g(t)\left((g-1)\log(v\tau_1+1-t)+\frac{t}{v\tau_1+1-t}\right)dt- \frac{u}{v}\xi_1f_g(\xi_1)\log\frac{v}{w}\\
    &<\frac{u}{v}f_g(v\tau_1)\int_{\xi_1}^{v\tau_1}\left((g-1)\log(v\tau_1+1-t)+\frac{t}{v\tau_1+1-t}\right)dt-\frac{u}{v}\xi_1f_g(\xi_1)\log\frac{v}{w}.
\end{align*}
The remaining integral is $\frac{v}{w}\left(g+\frac{w}{v}\xi_1\right)\log \frac{v}{w}-g \left(\frac{v}{w}-1\right)$, so that
\begin{equation}\label{I_1_bound}
  I_1<f_g(v\tau_1)\left(\frac{u}{w}\left(g+\frac{w}{v}\xi_1\right)\log \frac{v}{w}-g \left(\frac{u}{w}-\frac{u}{v}\right)-\frac{u}{v}\xi_1\frac{f_g(\xi_1)}{f_g(\tau_1 v)}\log \frac{v}{w}\right).
\end{equation}
For $I_2$, we make use of \eqref{f_diffeq} and integrate by parts to observe that
\begin{align*}
  I_2&=\int_{\xi_1}^{v\tau_1}\left(t^{g}f_g(t)\right)'\frac{dt}{gt^{g-1}(v\tau_1+1-t)}\\
    &=\frac{f_g(v\tau_1)}{g}v\tau_1- \frac{f_g(\xi_1)}{g}\xi_1\frac{w}{v}+\int_{\xi_1}^{v\tau_1}f_g(t)\left(\frac{g-1}{g(v\tau_1+1-t)}-\frac{t}{g(v\tau_1+1-t)^{2}}\right)dt.
\end{align*}
Since $f$ is increasing,
\begin{equation*}
    I_2\le\frac{f_g(v\tau_1)}{g}\left(v\tau_1- \frac{f_g(\xi_1)}{f_g(\tau_1 v)}\xi_1\frac{w}{v}+g\int_{\xi_1}^{v\tau_1}\left(\frac{g-1}{g(v\tau_1+1-t)}-\frac{t}{g(v\tau_1+1-t)^{2}}\right)dt\right).
\end{equation*}
The remaining integral is $\log \frac{v}{w}-\frac{v\tau_1+1}{g}\left(1- \frac{w}{v}\right)$, and so
\begin{equation}\label{I_2_bound}
  g \frac{u}{v}\left(\frac{v}{u}-\frac{v}{w}\right)I_2<f_{g}(\tau_1 v)\left(1- \frac{u}{w}\right)\left( \left(1- \frac{f_g(\xi_1)}{f_g(\tau_1 v)}\right)\xi_1\frac{w}{v}+g\log \frac{v}{w}\right).
\end{equation}
Inserting the bounds \eqref{I_1_bound} and \eqref{I_2_bound} into \eqref{I_sum} gives the stated lemma.
\end{proof}
\begin{lemma}\label{Main_J_Lemma}
Let $\xi_2:=v\tau_2+1- \frac{v}{u}$, and $0<\frac{1}{v}<\frac{1}{w}<\tau_1\le\frac{1}{2}<\frac{1}{u}<\tau_2\le 1$. If $\xi_2\ge\beta_{g+1}$, then
  \begin{equation*}
    \frac{v}{e^{\gamma}f_g(\tau_1 v)}J\le\frac{1}{f_g(\tau_1 v)}\frac{v}{e^{\gamma}}g\left(\log \frac{w}{u}-1+\frac{u}{w}\right)+\xi_2 \frac{g}{g+1}\frac{u}{e^{\gamma}}\frac{1-f_{g+1}(\xi_2)}{f_g(\tau_1 v)}\log \frac{v}{u}.
  \end{equation*}
\end{lemma}
\begin{proof}
Let $t-1=v(\tau_2-1/s)$, so $s=v/(v\tau_2+1-t)$. Under this change of variables,
\begin{align*}
  J&=g \frac{u}{v}\int_{\xi_2}^{v\tau_2+1- \frac{v}{w}}F_{g+1}(t-1) \left(\frac{v}{u}-v\tau_2-1+t\right)\frac{dt}{v\tau_2+1-t}\\
    &=-g \frac{u}{v}\int_{\xi_2}^{v\tau_2+1- \frac{v}{w}}F_{g+1}(t-1)(t-\xi_2)\ d\log \left(v\tau_2+1-t\right).
\end{align*}
Integrating by parts, and then using the fact that $F>1$, we have
\begin{multline*}
  J<g \frac{u}{v}\left[\int_{\xi_2}^{v\tau_2+1- \frac{v}{w}} \left(F_{g+1}(t-1)-F'_{g+1}(t-1)(t-\xi_2)\right)\log \left(v\tau_2+1-t\right)dt \right.\\
  - \left. \left(\frac{v}{u}-\frac{v}{w}\right)\log \frac{v}{w}\right].
\end{multline*}
Since $F$ is decreasing, $F'<0$, and
\begin{equation}\label{J_from_J_1}
    J<g \frac{u}{v}\left[J_1-\left(\frac{v}{u}-\frac{v}{w}\right)\log \frac{v}{w}\right],
\end{equation}
where
\begin{equation*}
  J_1=\int_{\xi_2}^{v\tau_2+1- \frac{v}{w}} F_{g+1}(t-1)\log \left(v\tau_2+1-t\right)dt.
\end{equation*}
Next, using \eqref{f_diffeq}, and assuming that $\xi_2\ge\beta_{g+1}$, we rewrite
\begin{align*}
    J_1&=\int_{\xi_2}^{v\tau_2+1- \frac{v}{w}}\left(t^{g+1}f_{g+1}(t)\right)'\log (v\tau_2+1-t)\frac{dt}{t^{g}(g+1)}.
\end{align*}
Integrating by parts, we find that
\begin{equation*}
J_1=\frac{(v\tau_2+1- \frac{v}{w})f_{g+1}(v\tau_2+1- \frac{v}{w})\log \frac{v}{w}}{(g+1)}-\frac{\xi_2 f_{g+1}(\xi_2)\log \frac{v}{u}}{(g+1)}+J_2,
\end{equation*}
where
\begin{equation*}
    J_2=-\int_{\xi_2}^{v\tau_2+1- \frac{v}{w}}t^{g+1}f_{g+1}(t)d\left(\frac{\log(v\tau_2+1-t)}{t^{g}(g+1)}\right).
\end{equation*}
Now, since $f<1$,
\begin{equation}\label{J_1_bound}
  J_1<\frac{(v\tau_2+1- \frac{v}{w})\log \frac{v}{w}}{(g+1)}-\frac{\xi_2 f_{g+1}(\xi_2)\log \frac{v}{u}}{(g+1)}+J_2,
\end{equation}
and
\begin{align*}
  J_2&=-\int_{\xi_2}^{v\tau_2+1- \frac{v}{w}}t^{g+1}f_{g+1}(t)d\left(\frac{\log(v\tau_2+1-t)}{t^{g}(g+1)}\right)\\
  &=\frac{1}{(g+1)}\int_{\xi_2}^{v\tau_2+1- \frac{v}{w}}f_{g+1}(t)\left\{g\log(v\tau_2+1-t)+\frac{t}{v\tau_2+1-t}\right\}dt\\
  &<\frac{1}{(g+1)}\int_{\xi_2}^{v\tau_2+1- \frac{v}{w}}\left\{g\log(v\tau_2+1-t)+\frac{t}{v\tau_2+1-t}\right\}dt.
\end{align*}
Calculating the remaining integral, we conclude that
\begin{equation}\label{J_2_bound}
  J_2<\frac{g}{g+1}(v\tau_2+1- \frac{v}{w})\log \frac{v}{w}-\frac{g}{g+1}\xi_2 \log \frac{v}{u}+(v\tau_2+1)\log \frac{w}{u}-\frac{v}{u}+\frac{v}{w}.
\end{equation}
Combining \eqref{J_2_bound} and \eqref{J_1_bound}, we have
\begin{multline*}
  J_1<(v\tau_2+1- \frac{v}{w})\log \frac{v}{w}-\frac{\xi_2}{g+1} f_{g+1}(\xi_2)\log \frac{v}{u}\\
  -\frac{g}{g+1}\xi_2 \log \frac{v}{u}+(v\tau_2+1)\log \frac{w}{u}- \frac{v}{u}+\frac{v}{w}.
\end{multline*}
Since $v\tau_2+1- \frac{v}{w}=\left(\frac{v}{u}-\frac{v}{w}\right)+(v\tau_2+1- \frac{v}{u})=\left(\frac{v}{u}-\frac{v}{w}\right)+\xi_2$, we conclude from \eqref{J_from_J_1} that
\begin{equation*}
  J<g \frac{u}{v}\left(-\frac{\xi_2}{g+1} f_{g+1}(\xi_2)\log \frac{v}{u}+\xi_2\log \frac{v}{w}-\frac{g}{g+1}\xi_2 \log \frac{v}{u}+(v\tau_2+1)\log \frac{w}{u}- \frac{v}{u}+\frac{v}{w}\right),
\end{equation*}
or equivalently,
\begin{equation*}
  J<g \frac{u}{v}\left(-\frac{\xi_2}{g+1} f_{g+1}(\xi_2)\log \frac{v}{u}+\xi_2\log\frac{v}{u}-\frac{g}{g+1}\xi_2\log\frac{v}{u}+\frac{v}{u}\log\frac{w}{u}-\frac{v}{u}+\frac{v}{w}\right).
\end{equation*}
Simplifying the right-hand side, this reads,
\begin{equation*}
  J<\frac{g}{g+1}\frac{u}{v}\xi_2\left(1-f_{g+1}(\xi_2)\right)\log\frac{v}{u}+g\left(\log\frac{w}{u}-1+\frac{u}{w}\right).
\end{equation*}
Multiplying this inequality by $\frac{v}{e^{\gamma}f_g(\tau_1 v)}$ gives the stated lemma.
\end{proof}

\section{Proof of Theorem 1}

Setting $\xi_1=\beta_g$ in  Lemma~\ref{Main_I_Lemma}, and $\xi_2=\beta_{g+1}$ in Lemma~\ref{Main_J_Lemma} gives
\begin{equation}\label{I_bound_version_3}
  \frac{1}{f_g(\tau_1 v)}I\le \left(g+\frac{u}{v}\beta_g\right)\log \frac{v}{w}+\frac{w}{v} \beta_g\left(1- \frac{u}{w}\right)-g \left(\frac{u}{w}- \frac{u}{v}\right),
\end{equation}
and
\begin{equation}\label{J_bound_intermediate_version}
    \frac{v}{e^{\gamma}f_g(\tau_1 v)}J\le\frac{1}{f_g(\tau_1 v)}\left(\frac{v}{e^{\gamma}}g\left(\log \frac{w}{u}-1+\frac{u}{w}\right)+\frac{\beta_{g+1}}{g+1}\frac{ug}{e^{\gamma}}\log \frac{v}{u}\right).
\end{equation}
Setting $\tau_1=1/2$ and $\tau_2=1$, this choice of $\xi_1$ and $\xi_2$ implies that
\begin{equation}\label{u_choice}
    u=1+\frac{\beta_{g+1}-1}{v-\left(\beta_{g+1}-1\right)},
\end{equation}
and
\begin{equation}\label{w_choice}
    w=2\left(1+\frac{2(\beta_g-1)}{v-2\left(\beta_g-1\right)}\right).
\end{equation}
The parameters $u$ and $w$ will therefore be completely determined by our choice of $v$. 

To simplify the analysis, we bound the ratio $w/u$, defined for $v>\max\{\beta_{g+1}-1,2(\beta_g-1)\}$. Let $N\ge 3$ be chosen so that $N(\beta_{g+1}-1)>\max\{\beta_{g+1}-1,2(\beta_g-1),4(\beta_g-1)-(\beta_{g+1}-1)\}$. Assuming that
\begin{equation}\label{v_assumption}
v\ge N(\beta_{g+1}-1),
\end{equation}
then
\begin{equation}\label{ratio_assumption}
\frac{4}{3}\le \frac{w}{u}\le 4. 
\end{equation}
The upper bound is easy to see since
\begin{equation*}
    \frac{w}{u}=\frac{2(v-(\beta_{g+1}-1))}{v-2(\beta_g-1)}\le 4
\end{equation*}
if $v\ge 4(\beta_g-1)-(\beta_{g+1}-1)$, which holds for \eqref{v_assumption}. Next, if $\max\{\beta_{g+1}-1,2(\beta_g-1)\}=\beta_{g+1}-1$,
\[
    g(v):=\frac{2(v-(\beta_{g+1}-1))}{v-2(\beta_g-1)}
\]
is an increasing function. Therefore, for $v$ satisfying \eqref{v_assumption},
\[
    g(v)\ge g\left(N(\beta_{g+1}-1)\right)=\frac{2(N-1)(\beta_{g+1}-1)}{N(\beta_{g+1}-1)-2(\beta_g-1)}\ge\frac{2(N-1)}{N}\ge\frac{4}{3}, 
\]
since $\beta_g\ge 1$, and $N\ge 3$. If, on the other hand, $\max\{\beta_{g+1}-1,2(\beta_g-1)\}=2(\beta_g-1)$, then
\[
    g(v)\ge 2,
\]
since this is equivalent to $\beta_{g+1}-1\le 2(\beta_g-1)$. In either case, the lower bound for $\frac{w}{u}=g(v)$ in \eqref{ratio_assumption} holds.

Using \eqref{ratio_assumption}, the bound for $J$ in \eqref{J_bound_intermediate_version} simplifies to
\begin{align}\label{J_bound_version_4}
  \frac{v}{e^{\gamma}f_g(\tau_1 v)}J&\le\left(\frac{v}{e^{\gamma}}g\left(\log 4-1+\frac{3}{4}\right)+\frac{\beta_{g+1}}{g+1}\frac{ug}{e^{\gamma}}\log \frac{v}{u}\right)\left(1+O\left(e^{-v/2}\right)\right)\\\notag
  &\le\frac{vg}{C_0}+\frac{\beta_{g+1}}{g+1}\frac{ug}{e^{\gamma}}\log\frac{v}{u}+O\left(\frac{vg}{e^{v/2}}+\frac{ug\log v}{e^{v/2}}\right),
\end{align}
where we have used the boundary condition in  \eqref{F_f_boundary_condition}, and defined
\begin{equation}\label{C_0_def}
    C_0:=\frac{e^{\gamma}}{\log 4 -\frac{1}{4}}.
\end{equation}
Ultimately, our choice of $v$ in \eqref{v_choice} will guarantee that the error term above is $o(1)$, and that our assumption that $v\ge N(\beta_{g+1}-1)$ in \eqref{v_assumption} is valid provided $k$ is sufficiently large, say
\begin{equation}\label{k_bound}
    k\ge\frac{(N-1)^2(\beta_{g+1}-1)}{C_0}.
\end{equation}

Lemma~\ref{Richert_Weight_Lemma}, Lemma~\ref{W_integral_form}, and \eqref{V_def} guarantee \eqref{P_r_bound} is satisfied provided we select an $r$ such that
\begin{equation}\label{r_requirement}
  r>gku-1+\frac{1}{f_{g}(\tau_1 v)}I(u,w,v)+\frac{v}{e^\gamma f_{g}(\tau_1 v)}J(u,w,v).
\end{equation}
Ignoring error terms, the bounds in \eqref{I_bound_version_3} and \eqref{J_bound_version_4} show that it is enough to select an $r$ such that
\begin{equation}\label{r_dirty}
    r>gku-1+\left(g+\frac{u}{v}\beta_{g}\right)\log\frac{v}{w}+\frac{w}{v}\beta_{g}\left(1-\frac{u}{w}\right)-g\left(\frac{u}{w}-\frac{u}{v}\right)+\frac{vg}{C_0}+\frac{\beta_{g+1}}{g+1}\frac{ug}{e^{\gamma}}\log\frac{v}{u}.
\end{equation}
In search of the smallest such $r$, we choose $v$ to minimize the expression on the right. For the sake of simplicity, we focus on the most problematic terms in this expression, given by
\begin{equation*}
  M(v):=gku+\frac{vg}{C_0}=gk+\frac{(\beta_{g+1}-1)gk}{v-(\beta_{g+1}-1)}+\frac{vg}{C_0},
\end{equation*}
where $C_0$ is defined in \eqref{C_0_def}, also keeping in mind \eqref{u_choice}. The minimum is achieved at 
\begin{equation}\label{v_choice}
  v=\beta_{g+1}-1+\sqrt{C_0(\beta_{g+1}-1)k},
\end{equation}
at which
\begin{equation}\label{Franze_Kao_Main_Term}
  M(v)=gk+gk \left(2\sqrt{(\beta_{g+1}-1)/(C_0k)}+(\beta_{g+1}-1)/(C_0k)\right),
\end{equation}
and the remaining terms in \eqref{r_dirty} are $O(g\log gk)$. Therefore, the admissible $r$ in \eqref{r_dirty} take the form
\[
    r > gk+c_1 g^{3/2}k^{1/2}+c_2 g^2+O\left(g\log gk\right),
\]
where
\[
    c_1=2\sqrt{\frac{\beta_{g+1}-1}{C_0 g}} \quad \text{and} \quad c_2=\frac{\beta_{g+1}-1}{C_0 g}.
\]
Both $c_1$ and $c_2$ are $O(1)$. Thus, our admissible $r$ take the form stated in \eqref{r_admissible_theorem_version}.

Before moving on, note that we have shown \eqref{r_admissible_theorem_version} for $k$ satisfying \eqref{k_bound}, but that we may need an even larger $k$ to guarantee that these admissible $r$ are asymptotically better than those in \eqref{classical_r}. In fact, the main term in \eqref{Franze_Kao_Main_Term} satisfies $M(v)< 2gk$ if
\begin{equation*}
  \frac{\beta_{g+1}-1}{C_0 k}< 3-2\sqrt{2}.
\end{equation*}
Therefore, we suppose that
\[
  k>\max\left\{\frac{(N-1)^2(\beta_{g+1}-1)}{C_0},\frac{\beta_{g+1}-1}{C_0(3-2\sqrt{2})}\right\}.
\]
However, numerical data suggests that the improvements appear much earlier.

For the admissible $r$-values in Table~\ref{T:1}, we briefly describe our choices of $v$, $w$, and $u$, for each fixed $g$ and $k$. All numerical experiments were conducted using W.\ Galway's \textit{Mathematica} package~\cite{Gal}. We chose the parameter $v$ to be of the form $v = \ga_g + n$, where $n$ is a positive integer. Next, we chose $w$ to minimize the expression on the right in \eqref{r_requirement}, which amounts to solving
\[
    F_g \left( v\left( \frac{1}{2} - \frac{1}{w} \right)\right) - \frac{v}{e^\gr} F_{g+1} \left(v\left(1 - \frac{1}{w}\right) \right)=0.
\]
With these choices of $v$ and $w$, we then chose $u$ to minimize the expression in \eqref{r_requirement} by solving
\[
    k f_g\left(\frac{v}{2}\right) - \int_w^v F_g\left(v\left(\frac{1}{2} - \frac{1}{s}\right)\right) \, \frac{ds}{s^2} - \frac{v}{e^\gr} \int_w^u F_{g+1}\left(v\left(1-\frac{1}{s}\right)\right) \, \frac{ds}{s^2} = 0.
\]
This process was repeated for many values of $n$ to arrive at the stated admissible $r$-values.

\section{Concluding Remarks}

More general results are readily available. For example, one could consider polynomials $H$ whose irreducible components have \emph{different} degrees. In addition, the work of Booker and Browning \cite{BB} allows one to capture \emph{squarefree} values, rather than almost-primes, if these irreducible components have degree 3 or less. The polynomial sequence considered here was chosen mainly for illustrative purposes.

\end{document}